\documentclass[12pt,leqno,oneside,a4paper]{amsart}
\usepackage{graphicx}
\usepackage{amscd}
\usepackage{amsmath}
\usepackage{amsfonts}
\usepackage{amssymb}
\usepackage{mathtools}
\usepackage{verbatim}
\usepackage{comment}
\usepackage{color}
\usepackage[usenames,dvipsnames,table,xcdraw]{xcolor}
\usepackage{pdfsync}
\usepackage{bbm}
\usepackage{dsfont}
\usepackage[colorlinks]{hyperref}
\usepackage{enumerate}
\usepackage{booktabs}
\usepackage{float}
\usepackage{wrapfig}

\usepackage{tikz}
\numberwithin{equation}{section}
\theoremstyle{plain}
\newtheorem{theorem}{Theorem}[section]
\newtheorem{lemma}[theorem]{Lemma}

\newtheorem{proposition}[theorem]{Proposition}

\theoremstyle{definition}

\theoremstyle{remark}

\newtheorem{claim}[theorem]{Claim}

\newtheorem{case[theorem]}{Case}

\definecolor{blue}{rgb}{0,0,1}

\definecolor{red}{rgb}{1,0,.2}

\DeclareSymbolFont{extraup}{U}{zavm}{m}{n}
\DeclareMathSymbol{\varheart}{\mathalpha}{extraup}{86}
\DeclareMathSymbol{\vardiamond}{\mathalpha}{extraup}{87}

\newcommand{\DD}{\mathcal{D}}

\date{{\today}}

\usepackage{xr}
\newcommand{\ii}{{\mathbf{i}}}
\newcommand{\jj}{{\mathbf{j}}}
\newcommand{\N}{{\mathbb{N}}}
\newcommand{\R}{{\mathbb{R}}}
\newcommand{\var}{{\mathrm{var}}}

\begin{document}

\pagestyle{myheadings}

\title[Dimension theory of dynamically defined function graphs]{Dimension Theory of some non-Markovian repellers Part II: Dynamically defined function graphs}

\author{Bal\'azs B\'ar\'any}
\address[Bal\'azs B\'ar\'any]{Budapest University of Technology and Economics, Department of Stochastics, MTA-BME Stochastics Research Group, P.O.Box 91, 1521 Budapest, Hungary}
\email{balubsheep@gmail.com}

\author{Micha\l\ Rams}
\address[Micha\l\ Rams]{Institute of Mathematics, Polish Academy of Sciences, ul. \'Sniadeckich 8, 00-656 Warszawa, Poland}
\email{rams@impan.pl}

\author{K\'aroly Simon}
\address[K\'aroly Simon]{Budapest University of Technology and Economics, Department of Stochastics, Institute of Mathematics, 1521 Budapest, P.O.Box 91, Hungary} \email{simonk@math.bme.hu}

\subjclass[2010]{Primary 28A80 Secondary 28A78}
\keywords{Self-affine measures, self-affine sets, Hausdorff dimension.}
\thanks{The research of B\'ar\'any and Simon was partially supported by the grant OTKA K123782. B\'ar\'any acknowledges support also from NKFI PD123970 and the J\'anos Bolyai Research Scholarship of the Hungarian Academy of Sciences. Micha\l\ Rams was supported by National Science Centre grant 2014/13/B/ST1/01033 (Poland). This work was partially supported by the  grant  346300 for IMPAN from the Simons Foundation and the matching 2015-2019 Polish MNiSW fund.}

\begin{abstract}
This is the second part in a series of two papers. Here, we give an overview on the dimension theory of some dynamically defined function graphs, like Takagi and Weierstrass function, and we study the dimension of markovian fractal interpolation functions and generalised Takagi functions generated by non-Markovian dynamics.
\end{abstract}

\date{\today}
\maketitle


\section{The Weierstrass and Takagi functions}

The study of the geometric properties of the graphs of real functions goes back to the 19th century. Karl Weierstrass introduced in 1872 a function, which is continuous but nowhere differentiable. That was one of the first examples of such functions and for nowadays, became a famous example:
\begin{equation}\label{eq:weier}
W_{\alpha,b}(x)=\sum_{n=0}^\infty\alpha^n\cos(2\pi b^nx),
\end{equation}
where $b>1$ and $1/b<\alpha<1$. In fact, Weierstrass proved the non-differentiability for some values of parameters, and the proof for all parameters was given by Hardy \cite{Hardy} in 1916.

Teiji Takagi \cite{takagi} published his simple example of a continuous but nowhere differentiable function in 1901,
\begin{equation}\label{eq:takagi}
T(x)=\sum_{n=0}^\infty 2^{-n}\psi(2^nx),
\end{equation}
where $\psi(x)=\mathrm{dist}(x,\mathbb{Z})$. Unlike for the Weierstrass function, it is easy to show that $T$ has at no point a finite derivative, which proof is due to  Billingsley \cite{Billi}. For further properties and historical background of the functions above, see the survey papers of Allaart and Kawamura \cite{AllKawa} and Bara\'nski~\cite{baranski}.

Later, starting from the work of Besicovitch and Ursell \cite{BesUr}, the graphs of $W_{\alpha,b}$ and related functions were studied from a geometric point of view as fractal curves in the plane. In general, let
\begin{equation}\label{eq:gen}
G_{\alpha,b}(x)=\sum_{n=0}^\infty \alpha^{n}\phi(b^nx)
\end{equation}
for $x\in\mathbb{R}$, where $b\in\N$, $1/b<\alpha<1$ and $\phi\colon\mathbb{R}\mapsto\mathbb{R}$ is a non-constant $\mathbb{Z}$-periodic Lipschitz continuous piecewise $C^1$ function. Kaplan, Mallet-Paret and Yorke~\cite{KMY} proved that a function of the form \eqref{eq:gen} is either piecewise $C^1$ smooth or the box dimension of the graph is equal to
\begin{equation}\label{eq:lyapdimspec}
D=2+\frac{\log\alpha}{\log b}.
\end{equation}
This fact is related to the H\"older continuity of the function $G_{\alpha,b}$. In fact, if the function $g\colon[0,1]\mapsto\mathbb{R}$ is H\"older continuous with exponent $\alpha$ then
$$
\overline{\dim}_B\{(x,g(x)): x\in[0,1]\}\leq 2-\alpha.
$$
For instance, the case of smoothness of $G_{\alpha,b}$ happens if $\phi(x)=\alpha h(b x)-h(x)$ for some smooth function $h$.

The problem of determining the value of the Hausdorff dimension turned out to be much more complicated. Mandelbrot formulated the conjecture in 1977 \cite{Mandel} that the Hausdorff dimension of the graph of $W_{\alpha,b}$ equals to $D$, but this has been solved only recently.

Ledrappier \cite{Ledrapp} gave a sufficient condition in order to determine the Hausdorff dimension of the graph. In details, let $\underline{\xi}=\{\xi_i,i=1,2,\ldots\}$ be a sequence of independent Bernoulli variables with values $0,\ldots,b-1$ and with probabilities $\mathbb{P}(\xi_j=k)=1/b$. If the distribution of the random variable
\begin{equation}\label{eq:condled}
Y_x(\underline{\xi})=\sum_{n=1}^\infty(b\alpha)^{-n}\phi'\left(\frac{x}{b^n}
+\frac{\xi_1}{b^n}+\frac{\xi_2}{b^{n-1}}+\cdots++\frac{\xi_n}{b}\right)
\end{equation}
has dimension $1$ for Lebesgue almost every $x$ then $$\dim_H\{(x,G_{\alpha,b}(x)):x\in[0,1]\}=D.$$
This condition relies on the so-called Ledrappier-Young formula.

Altough, for the first sight this condition may seem very restrictive, it turned out that it widely applicable. In the case of Weierstrass functions \eqref{eq:weier}, Bara\'nski, B\'ar\'any and Romanowska \cite{BBR} showed that for every $b\geq2$ integers there exists $\alpha_b\in[1/b,1)$ such that for every $\alpha\in(\alpha_b,1)$,
$$
\dim_H\{(x,W_{\alpha,b}(x)):x\in[0,1]\}=D.
$$
Recently, Shen \cite{Shen} proved that $\alpha_b=1/b$.

In the case of Takagi function, the distribution of the random variable $Y_x(\underline{\xi})$ is independent of $x$ and it is the Bernoulli convolution, related to Erd\H{o}s' problem \cite{Erdos1,Erdos2}. For simplicity denote $T_{\alpha}$ the function $G_{\alpha,2}$ with $\psi(x)=\mathrm{dist}(x,\mathbb{Z})$. It is easy to see that $Y_x(\underline{\xi})=\sum_{n=0}^\infty(\delta_{\xi_n,0}-\delta_{\xi_n,1})(2\alpha)^{-n}$ in \eqref{eq:condled}, where $\delta_{i,j}=1$ if $i=j$ and $0$ otherwise. Using this phenomena, Solomyak \cite{Solomyak98} showed that for Lebesgue almost every $\alpha\in(1/2,1)$,
\begin{equation}\label{eq:takadim}
\dim_H\{(x,T_{\alpha}(x)):x\in[0,1]\}=D.
\end{equation}
Applying the result of Hochman \cite{hochman2012self},
\cite[Theorem~4.11]{BaRaSiKochin1}, there exists a set $E\subset(1/2,1)$ such that $\dim_HE=0$ and for every $\alpha\in(1/2,1)\setminus E$, \eqref{eq:takadim} holds. Recently, Varj\'u \cite{varju2018dimension} showed that the distribution of $Y(\underline{\xi})$ has dimension $1$ if $(2\alpha)^{-1}$ is a transcendental number (which is transcendental if and only if $\alpha$ is transcendental), and hence \eqref{eq:takadim} holds.

However, it is well known that for Pisot numbers (for instance $(2\alpha)^{-1}=(\sqrt{5}-1)/2$ the golden ratio) the distribution of $Y(\underline{\xi})$ is singular and has dimension strictly smaller than 1 and thus, Ledrappier's condition \eqref{eq:condled} cannot be applied. Recently, with different method, B\'ar\'any, Hochman and Rapaport \cite{BHR} proved that \eqref{eq:takadim} holds for every $\alpha\in(1/2,1)$.

\section{Dynamically defined function graphs}

Let $G_{\alpha,b}$ be the function defined in \eqref{eq:gen} with $b > 1$ integer, $1/b<\alpha<1$ and $\phi\colon\mathbb{R}\mapsto\mathbb{R}$ is a non-constant 1-periodic Lipschitz continuous piecewise $C^1$ function. It is easy to see that $G_{\alpha,b}$ satisfies certain self-similarity equation
\begin{equation}\label{eq:selfsim}
G_{\alpha,b}(x)=\alpha G_{\alpha,b}(bx)+\phi(x).
\end{equation}
Since $\phi$ is $1$-periodic and thus, $G_{\alpha,b}$ as well, \eqref{eq:selfsim} implies that $\mathrm{graph}(G_{\alpha,b})=\{(x,G_{\alpha,b}(x)):x\in[0,1]\}$ is invariant with respect to the dynamics
$$
F(x,y)=\left(bx\mod1,\frac{y-\phi(x)}{\alpha}\right)\text{ for }(x,y)\in[0,1]\times\mathbb{R},
$$
and $\{F^n(x,y)\}$ is bounded if and only if $y=G_{\alpha,b}(x)$.

One can define the local inverses of $F$ such that
$$
\tilde{F}_i(x,y)=\left(\frac{x+i}{b},\alpha y+\phi\left(\frac{x+i}{b}\right)\right)\text{ for }i=0,\ldots,b-1.
$$
Hence, $\mathrm{graph}(G_{\alpha,b})=\bigcup_{i=0}^{b-1}\tilde{F}_i(\mathrm{graph}(G_{\alpha,b}))$. For a visualisation of the local inverses in the cases of $W_{1/2,3}$ and $T_{2/3}$, see Figure~\ref{fig:takwei}.

\begin{figure}\label{fig:takwei}
  \centering
  \includegraphics[width=6cm]{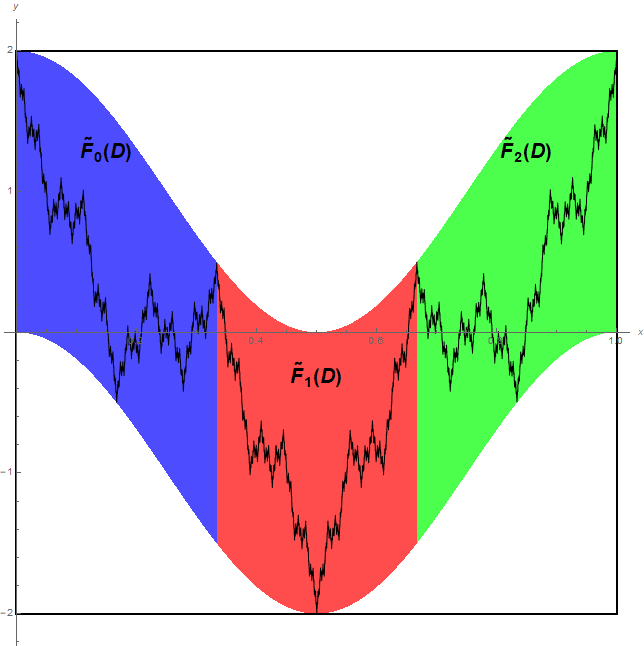}\qquad
\includegraphics[width=6cm]{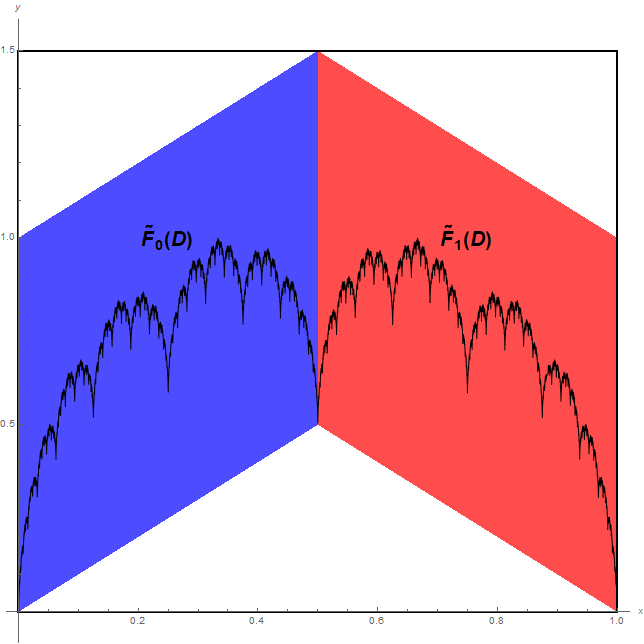}
\caption{The graph of $W_{1/2,3}$ and $T_{2/3}$ as repellers.}
\end{figure}

Observe that for the Takagi function $T_\alpha$, the function $\phi$ is piecewise linear, moreover, the singularity occur exactly at $x=1/2$. Thus, $\mathrm{graph}(T_\alpha)$ is a self-affine set, see
\cite[Definition~6.1]{BaRaSiKochin1}, with IFS
$$
\left\{\tilde{F}_0(\underline{x})=\begin{pmatrix}
                   \frac{1}{2} & 0 \\
                   1 & \alpha
                 \end{pmatrix}\underline{x}, \tilde{F}_1(\underline{x})=\begin{pmatrix}
                   \frac{1}{2} & 0 \\
                   -1 & \alpha
                 \end{pmatrix}\underline{x}+\begin{pmatrix}
                                              \frac{1}{2} \\
                                              1
                                            \end{pmatrix}\right\}
$$
formed by lower triangular matrices.

A wider family of continuous functions, which are attractors of affine IFS, is the fractal interpolation functions, introduced by Barnsley \cite{Barnsley}. Let a data set $\Delta=\{(x_i, y_i)\in[0,1]\times\mathbb{R}: i=0,1,\ldots,m\}$ be given so that $0=x_0<x_1<\cdots<x_{m-1}<x_m=1$. We concern the graphs of continuous functions $G\colon[0,1]\mapsto\mathbb{R}$, which interpolate the data according to $G(x_i)=y_i$ for $i\in\{0, 1,\ldots,m\}$, and $\mathrm{graph}(G)$ is the attractor of an IFS, which contains only affine transformations with lower triangular matrices. That is,
$$
\left\{\tilde{F}_i\binom{x}{y}=\binom{(x_i-x_{i-1})x+x_{i-1}}{(y_i-y_{i-1}-\alpha_i(y_m-y_0))x+\alpha_iy+y_{i-1}-\alpha_iy_0}\right\}_{i=1}^m
$$
where $\alpha_i\in(-1,1)\setminus\{0\}$ are free parameters for $i=1,\ldots,m$. In other words, the interpolation function $G$ is the repeller of the piecewise linear, expanding map $F$, where $F(x,y)=F_i(x,y)$ if $x_{i-1}<x<x_i$ and
\begin{equation}\label{eq:gendyn}
F_i(x,y)=\left(\frac{x-x_{i-1}}{x_i-x_{i-1}},\frac{y-(y_i-y_{i-1}-\alpha_i(y_m-y_0))\frac{x-x_{i-1}}{x_i-x_{i-1}}-y_{i-1}+\alpha_iy_0}{\alpha_i}\right).
\end{equation}
For a visualisation of a fractal interpolation function, see Figure \ref{fig:fracint}.

\begin{figure}\label{fig:fracint}
  \centering
  \includegraphics[height=6cm]{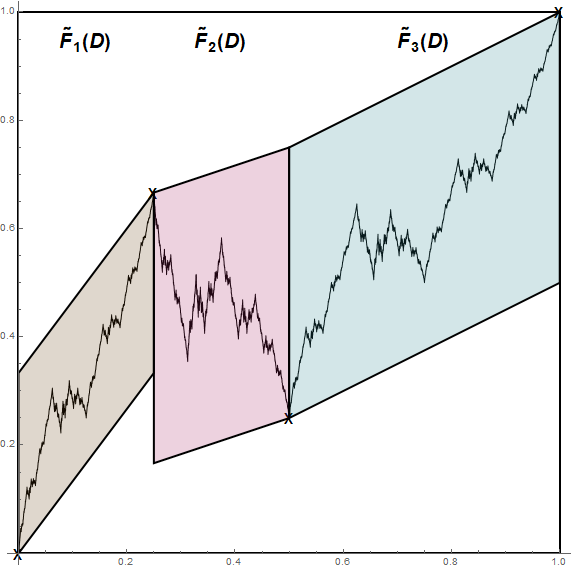}
\caption{The fractal interpolation function and its defining dynamics for $\Delta=\{(0,0),(1/4,2/3),(1/2,1/4),(1,1)\}$ and $\alpha_1=1/3, \alpha_2=-1/2$ and $\alpha_3=1/2$.}
\end{figure}

Note that if $\Delta$ is collinear then $G_{\underline{\alpha},\Delta}$ is a linear function and thus, its graph has dimension $1$. Thus, without loss of generality, the non-collinearity of $\Delta$ might be assumed without loss of generality.

Let us introduce the notation $G_{\underline{\alpha},\Delta}$, which denotes the fractal interpolation function for the data set $\Delta$ and free parameters $\underline{\alpha}\in\big((-1,1)\setminus\{0\}\big)^{|\Delta|-1}$.

Barnsley and Harrington \cite{BarnsleyHarrington} calculated the box dimension of $\mathrm{graph}(G)$ in a special case. Namely, when $x_i-x_{i-1}=1/m$ and $\alpha_i=\alpha$ for every $i=1,\ldots,m$ with $1/m<\alpha$, and the data is not situated on a line. Note that in this case the interpolation function corresponds to $G_{\alpha,m}$ in \eqref{eq:gen} with
\begin{multline}\label{eq:fracintphi}
\phi(x)=(y_i-y_{i-1}-\alpha(y_m-y_0))\left(mx+\frac{y_{i-1}-\alpha y_0}{y_i-y_{i-1}-\alpha(y_m-y_0)}-(i-1)\right)\\
\text{if }\frac{i-1}{m}\leq x<\frac{i}{m}.
\end{multline}
In this case,$$\dim_B\mathrm{graph}(G_{\alpha,m})=2+\frac{\log\alpha}{\log m}.$$

This result was later generalised by Bedford~\cite{Bedford} for general $\alpha_i$ but with the assumption that $x_i-x_{i-1}=1/m$ with $\alpha_i>1/m$ for every $i=1,\ldots,m$. Ruan, Su and Yao \cite{RuanSuYao} studied the box dimension in further generality. The complete characterization of the box counting dimension follows by Falconer and Miao \cite[Corollary~3.1]{FalconerMiao}. Namely, if $\Delta$ is not collinear then
$$
\dim_B\mathrm{graph}(G_{\underline{\alpha},\Delta})=\begin{cases}1 & \text{if }\sum_{i=1}^m|\alpha_i|\leq1\text{ and }\\
s &\text{if }\sum_{i=1}^m|\alpha_i|>1,
\end{cases}
$$
where $\sum_{i=1}^m|\alpha_i|(x_i-x_{i-1})^{s-1}=1$.

The following extension for the Hausdorff dimension follows by B\'ar\'any, Hochman and Rapaport \cite{BHR}.

\begin{theorem}
  Let the data set $\Delta=\{(x_i, y_i)\in[0,1]\times\mathbb{R}: i=0,1,\ldots,m\}$ be given so that $0=x_0<x_1<\cdots<x_{m-1}<x_m=1$. If $\sum_{i=1}^m|\alpha_i|>1$ and there exists $i\neq j$ such that
  \begin{equation}\label{eq:condforHD}
  \frac{y_i-y_{i-1}-\alpha_i(y_m-y_0)}{x_i-x_{i-1}-\alpha_i}\neq\frac{y_j-y_{j-1}-\alpha_j(y_m-y_0)}{x_j-x_{j-1}-\alpha_j}
  \end{equation}
  then
    $$
    \dim_H\mathrm{graph}(G_{\underline{\alpha},\Delta})=s,\text{ where }\sum_{i=1}^m|\alpha_i|(x_i-x_{i-1})^{s-1}=1.
    $$
\end{theorem}

The assumption \eqref{eq:condforHD} is a little bit stronger than non-collinearity of $\Delta$. That is, if $\Delta$ is collinear then \eqref{eq:condforHD} does not hold. The condition \eqref{eq:condforHD} is equivalent with the condition that the matrices $\{DF_i\}_{i=1}^m$ are not simultaneously diagonalisable.

Note that \eqref{eq:condforHD} is a milder condition than Ledrappier's condition \eqref{eq:condled}. For example, suppose that the fractal interpolation function corresponds to a function of the form \eqref{eq:gen} with a $1$-periodic piecewise linear $\phi$. That is, the data set $\Delta=\{(\frac{i}{m},y_i):i=0,\ldots,m\}$, $y_0=y_m=0$ and $\alpha_1=\cdots=\alpha_m=\alpha$. Then $\phi$ is the piecewise linear function, connecting the data set $\Delta$, i.e.
$$
\phi(x)=(y_i-y_{i-1})(mx-(i-1))+y_{i-1}\text{ if }\frac{i-1}{m}\leq x<\frac{i}{m}
$$
for $i=1,\ldots,m$. Then \eqref{eq:condled} has the form
$$
Y(\underline{\xi})=m\sum_{n=1}^\infty(m\alpha)^{-n}(y_{\xi_n}-y_{\xi_{n}-1}),
$$
where $\{\xi_n\}$ are independent random variables with $\mathbb{P}(\xi_i=k)=1/m$ for $k=1,\ldots,m$. Ledrappier's condition requires that the distribution of the random variable $Y$ has dimension $1$ but the condition \eqref{eq:condforHD}, i.e. $y_i-y_{i-1}\neq y_j-y_{j-1}$ for some $i\neq j$, is equivalent to that the distribution of the random variable $Y$ has positive dimension.

\section{Markovian fractal interpolation functions}

Let $\Delta=\{(x_i, y_i)\in[0,1]\times\mathbb{R}: i=0,1,\ldots,m\}$ be given so that $0=x_0<x_1<\cdots<x_{m-1}<x_m=1$, and let $\alpha_i\in(-1,1)\setminus\{0\}$ for $i=1,\ldots,m$. The expanding dynamics, of which repeller is $\mathrm{graph}(G_{\underline{\alpha},\Delta})$, has a skew product form. That is, the map $F(x,y)$ has the form
\begin{equation}\label{eq:interpoldyn}
F(x,y)=F_i(x,y)=(f_i(x),g_i(x,y))\text{ for }x\in(x_{i-1},x_i).
\end{equation}
Thus, there is a base dynamics $f\colon[0,1]\mapsto[0,1]$, which is a piecewise linear, expanding interval map. In particular, each subinterval $(x_{i-1},x_i)$ is mapped to the complete interval $(0,1)$. A natural generalisation could be when the base dynamics $f$ is a Markovian expanding map with Markov partition $\{(x_{i-1},x_i):i=1,\ldots,m\}$.

That is, for every $i=1,\ldots,m$ let $0\leq\ell(i)<r(i)\leq m$ be integers such that $\gamma_i:=\frac{x_{r(i)}-x_{\ell(i)}}{x_i-x_{i-1}}>1$. Then let
$$
f(x)=f_i(x):=\frac{x_{r(i)}-x_{\ell(i)}}{x_i-x_{i-1}}(x-x_{i-1})+x_{\ell(i)}\text{ if }x\in(x_{i-1},x_i).
$$
By the choice of $\ell(i),r(i)$, the map $f$ is a piecewise linear expanding Markov map, see 
\cite[Definition~10.1]{BaRaSiKochin1}.

For each $i=1,\ldots,m$, let $\alpha_i\in(-1,1)\setminus\{0\}$ be arbitrary. Then let $g_i(x,y)$ be of the form $g_i(x,y)=\lambda_iy+a_ix+t_i$ such that $\lambda_i=\alpha_i^{-1}$,  $g_i(x_{i-1},y_{i-1})=y_{\ell(i)}$ and $g_i(x_{i},y_{i})=y_{r(i)}$. This assumption guarantees that the repeller of $F$ in \eqref{eq:interpoldyn} is a graph of a function $G$ so that $G(x_i)=y_i$ for $i=0,\ldots,m$. Simple calculations show that
$$
a_i=\frac{y_{r(i)}-y_{\ell(i)}-\alpha_i^{-1}(y_i-y_{i-1})}{x_i-x_{i-1}}\text{ and }t_i=y_{\ell(i)}-\alpha_i^{-1}y_{i-1}-a_ix_{i-1}.
$$
For a visualisation of a markovian fractal interpolation function, see Figure~\ref{fig:fracint2}.

\begin{figure}\label{fig:fracint2}
  \centering
  \includegraphics[width=7cm]{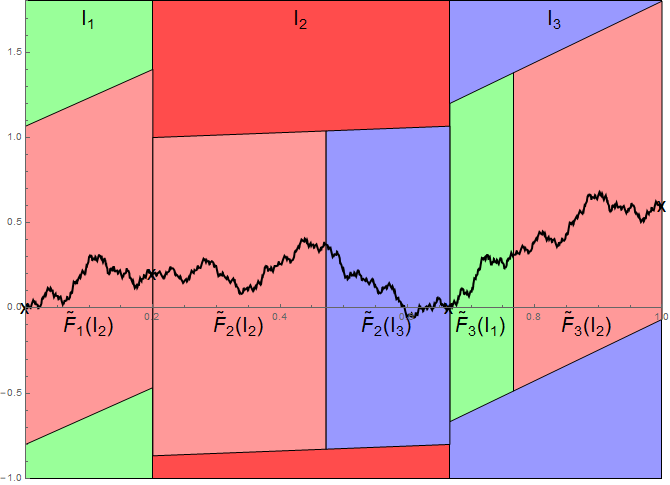}
\caption{The markovian fractal interpolation function and its defining dynamics for $\Delta=\{(0,0),(1/5,1/5),(2/3,0),(1,3/5)\}$ and $\alpha_1=2/3, \alpha_2=-2/3$ and $\alpha_3=2/3$.}
\end{figure}

\begin{figure}\label{fig:fracint3}
	\centering
	\includegraphics[width=11cm]{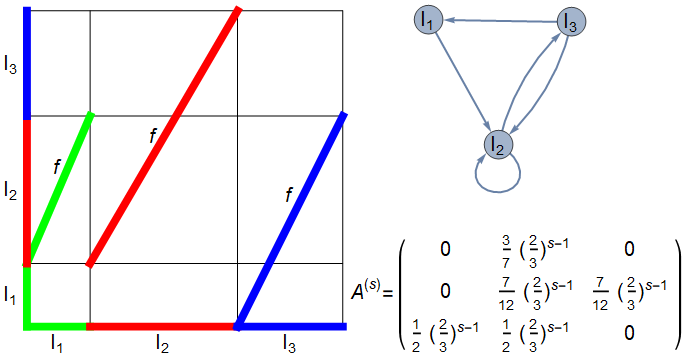}
	\caption{The base system $f$, its markovian structure and the matrix $A^{(s)}$ of the markovian fractal interpolation function of Figure~\ref{fig:fracint2}.}.
\end{figure}

Since the base dynamics is Markov, not all sequences of functions $f_i$ is admissible. We define the following $m\times m$ matrix $A=(A_{i,j})_{i,j=1}^m$ as follows
\begin{equation}\label{eq:adjacency}
A_{i,j}=\begin{cases}
1 & \text{if }\ell(i)+1\leq j\leq r(i), \\
0 & \text{otherwise. }
\end{cases}
\end{equation}
Hence, an infinite sequence $\mathbf{i}=(i_1,i_2,\ldots)$ is addmissible if $A_{i_k,i_{k+1}}=1$ for every $k=1,2,\ldots$. Denote $\Sigma_A\subseteq\{1,\ldots,m\}^{\mathbb{N}}$ the set of all admissible sequences, that is, $\ii=(i_1,i_2,\ldots)\in\Sigma_A$ if and only if $A_{i_k,i_{k+1}}=1$ for every $k\geq1$. By using the local inverses $\tilde{F}_i$, one can define the natural map from $\Sigma_A$ to $\mathrm{graph}(G)$ as
\begin{equation}\label{eq:natproj}
\Pi(\mathbf{i})=\lim_{n\to\infty}\tilde{F}_{i_1}\circ\cdots\circ\tilde{F}_{i_n}(x_{\ell(i_n)},y_{\ell(i_n)}).
\end{equation}
Thus, $\Pi(\mathbf{i})_2=G(\Pi(\mathbf{i})_1)$, where $\Pi(\mathbf{i})_i$ denotes the $i$th coordinate of $\Pi(\mathbf{i})$, moreover, $F(\Pi(\mathbf{i}))=\Pi(\sigma\mathbf{i})$, where $\sigma$ is the left-shift on $\Sigma_A$.

Since $f$ is Markov with respect to the intervals $\{[x_{i-1},x_i]\}_{i=1}^m$, one can decompose the intervals into finitely many classes with respect to recurrency. Since the repeller of $F$ restricted to any recurrent class of intervals is $\mathrm{graph}(G)$ restricted to the intervals, without loss of generality, we may assume that $f$ is topologically transitive. On the other hand, if the period of $f$ would be $p\geq2$ then again by decomposing  the intervals into finitely many classes, the repeller of $F^p$ restricted to a class is the restriction of $\mathrm{graph}(G)$. Thus, without loss of generality, we may assume that $f$ (and the matrix $A$) is aperiodic. Namely, there exists a positive $k\geq1$ such that every element of $A^k$ is positive.

Since the local inverses are strict contractions, there exists an interval $D=[a,b]$ such that $\bigcup_{i=1}^m\tilde{F}_i([x_{\ell(i)},x_{r(i)}]\times D)\subseteq [0,1]\times D$. In order to determine the box counting dimension of $\mathrm{graph}(G)$, it is natural to cover $\mathrm{graph}(G)$ with sets of the form $\tilde{F}_{\pmb\omega}([x_{\ell(i_{|\pmb\omega|})},x_{r(i_{|\pmb\omega|})}]\times D)$. These sets are paralelograms with height paralel to the $x$-axis $\gamma_{\pmb\omega}$ and side length (paralel to the $y$-axis) $\alpha_{\pmb\omega}$.

Let us define the matrix $A^{(s)}=(A_{i,j}^{(s)})_{i,j=1}^m$ for $s\in[1,2]$ as follows
\begin{equation}\label{eq:smatrix}
A_{i,j}^{(s)}=|\alpha_i|\gamma_i^{-(s-1)}A_{i,j}=\begin{cases}
|\alpha_i|\gamma_i^{-(s-1)} & \text{if }\ell(i)+1\leq j\leq r(i), \\
0 & \text{otherwise. }
\end{cases}
\end{equation}

Similarly to Barnsley's fractal interpolation function, we distinguish two cases $\rho(A^{(1)})\leq1$ and $\rho(A^{(1)})>1$, where $\rho(\cdot)$ denotes the spectral radius. The first case implies that for most of the sets $\tilde{F}_{\pmb\omega}([x_{\ell(i_{|\pmb\omega|})},x_{r(i_{|\pmb\omega|})}]\times D)$, the component on the $x$-axis is longer than the component on the $y$-axis.

\begin{theorem}\label{thm:markovbox} If the data set $\Delta$ is not collinear then
\begin{equation}
\dim_B\mathrm{graph}(G)=\begin{cases}
1 & \text{if }\rho(A^{(1)})\leq1,\\
s & \text{if }\rho(A^{(1)})>1,
\end{cases}
\end{equation}
where $s$ is the unique solution of the equation $\rho(A^{(s)})=1$.
\end{theorem}

For completeness, we give a proof later.

The problem of Hausdorff dimension is significantly different. In point of view of Theorem~\ref{thm:markovbox}, it is natural to assume that $\rho(A^{(1)})>1$. One way to find the Hausdorff dimension of $\mathrm{graph}(G)$ is to find a iterated function system of affine transformations, which attractor is contained in $\mathrm{graph}(G)$, and satisfies the conditions given in  B\'ar\'any, Hochman and Rapaport \cite{BHR},
\cite[Theorem~6.3]{BaRaSiKochin1}.

\begin{theorem}\label{thm:markovhaus} Let the data set $\Delta$ be not collinear, the adjacency matrix $A$ be irreducible and aperiodic, and $(\alpha_1,\ldots,\alpha_m)\in\left((-1,1)\setminus\{0\}\right)^m$ be such that $\rho(A^{(1)})>1$. Moreover, let us assume that there exist $\ell\geq1$, $\pmb\omega,\pmb\tau\in\Sigma_{A,\ell}$ such that
\begin{equation}\label{eq:assum}
\alpha_{\pmb\omega}=\alpha_{\pmb\tau}, \gamma_{\pmb\tau}=\gamma_{\pmb\omega}, \omega_1=\tau_1, \omega_\ell=\tau_\ell\text{ and }D\tilde{F}_{\pmb\omega}\neq D\tilde{F}_{\pmb\tau}.
\end{equation}
Then
$$
\dim_H\mathrm{graph}(G)=s,\text{ where $s$ is the unique solution of $\rho(A^{(s)})=1$.}
$$
\end{theorem}

We remind that $\Sigma_n=\{1,\ldots,m\}^n$ is the collection of words of length $n$.
For $n\in\N$, let $(p_1,\ldots,p_{|\Sigma_n|})$ be a probabiliy vector and let $\nu$ be the corresponding Bernoulli measure, living on $(\Sigma_n^\N,\sigma_{\Sigma_n})$, where $\sigma_{\Sigma_n}$ is the usual left shift but acting on $\Sigma_n^\N$. We have a natural isometry between $(\Sigma_n^\N,\sigma_{\Sigma_n})$ and $(\Sigma, \sigma^n)$, let $\tilde\nu$ be the image of $\nu$ under this isometry. Finally, let
\[
\hat\nu = \frac 1n \sum_{i=0}^{n-1} \tilde\nu \circ \sigma^{-i}.
\]
The measures $\hat\nu$ that can be obtained by this construction will be called $n$-{\it Bernoulli} measures. Note that the $n$-Bernoulli measures are ergodic and $\sigma$ invariant measures on $\Sigma$.

\begin{proposition} \label{prop:nbern}
Let $A$ be an irreducible and aperiodic adjacency and let $(\Sigma_A, \sigma)$ be a subshift of finite type and let $\mu$ be a $\sigma$-invariant measure supported on $\Sigma_A$. Then there exists a sequence of $n$-Bernoulli measures $\hat\nu_n, n\to\infty$ supported on $\Sigma_A$ and converging to $\mu$ both in weak-* topology and in entropy.
\end{proposition}
\begin{proof}

Fix $k$ such that all elements of $A^k$ are positive. We choose a pair $(i,j)\in \{1,\ldots,m\}^2$ such that $A_{ij}=1$. For every $\ell\in \{1,\ldots,m\}$ we can choose a word $\pmb{p}(\ell)\in\Sigma_{A,k}$ such that $p_1=j$ and $\pmb{p}(\ell)\ell\in\Sigma_{A,k+1}$ and a word $\pmb{s}(\ell)\in\Sigma_{A,k}$ such that $s_k=i$ and $\ell\pmb{s}(\ell)\in\Sigma_{A,k+1}$. For any $n\geq 2k+1$ and for any word $\pmb\omega\in\Sigma_{A, n-2k}$ let $\hat{\pmb\omega}=\pmb{p}(\omega_1)\pmb{\omega}\pmb{s}(\omega_{n-2k})$, denote the set of such words by $\hat\Sigma_{A,n}$. Note that $\hat\Sigma_{A,n}\subset \Sigma_{A,n}$, moreover each word $\hat{\pmb\omega}$ begins with $j$ and ends with $i$, hence any concatenation of those words is also admissible.

Let us show this construction on the example in Figure~\ref{fig:fracint2}. In this case $$
A=\begin{pmatrix}
    0 & 1 & 0 \\
    0 & 1 & 1 \\
    1 & 1 & 0
  \end{pmatrix}.
$$
Choose $(i,j)=(2,2)$. The matrix $A^3$ has strictly positive elements, and it is easy to check that choices $\pmb{p}(1)=(2,3),\pmb{p}(2)=(2,2),\pmb{p}(3)=(2,2)$ and $\pmb{s}(1)=(2,2),\pmb{s}(2)=(2,2),\pmb{s}(3)=(2,2)$ are admissible and appropriate.

Let $\nu_n$ be the the Bernoulli measure on $(\Sigma_{A,n}^\N,\sigma_{\Sigma_n})$ obtained by the probabiliy vector $(p_{\pmb\tau})_{\pmb\tau\in\Sigma_{A,n}}$, where
\[
p_{\pmb\tau} =
\begin{cases}
\mu([\pmb\omega])&\text{if there exists $\pmb\omega\in\Sigma_{A,n-2k}$ such that }\pmb\tau=\hat{\pmb\omega},\\
0&{\rm otherwise.}
\end{cases}
\]
Let $\tilde\nu_n$ be the measure on $(\Sigma_A,\sigma^n)$ and let $\hat\nu_n$ be the $n$-Bernoulli measure on $(\Sigma_A,\sigma)$ as introduced previously. We need to prove two claims.

\begin{claim}\label{claim:1} $h(\hat\nu_n)\to h(\mu)$ as $n\to\infty$.\end{claim}
\begin{proof}
We have
\[
h(\mu)=\lim_{n\to\infty} -\frac 1 {n-2k} \sum_{\Sigma_{A,n-2k}} \mu([\pmb\tau]) \log \mu([\pmb\tau]).
\]
At the same time,
\[
h(\tilde\nu_n,\sigma^n)= -\sum p_{\pmb\omega} \log p_{\pmb\omega}= -\sum_{\Sigma_{A,n-2k}} \mu([\pmb\tau]) \log \mu([\pmb\tau]),
\]
hence
\[
h(\hat\nu_n) = -\frac 1n \sum_{\Sigma_{A,n-2k}} \mu([\pmb\tau]) \log \mu([\pmb\tau]).
\]
\end{proof}
\begin{claim}\label{claim:2} $\hat\nu_n\to \mu$ in weak-* topology. \end{claim}

\begin{proof} Let $w\colon\Sigma\to\R$ be a continuous function and denote by $\mathrm{var}_\ell(w)$ the supremum of differences $w(x)-w(y)$ over $x,y$ belonging to the same $\ell$-th level cylinder. We have
\[
\left|\int w d\mu - \int w d\hat\nu_n\right| \leq \frac 1n \sum_{i=0}^{n-1} \left|\int w d(\mu\circ\sigma^{-i}) - \int w d(\tilde\nu_n\circ\sigma^{-i})\right|
\]
(we remind that $\mu$ is $\sigma$-invariant, hence $\mu=\mu\circ\sigma^{-i}$ for any $i\geq1$). For any $n-2k$-th level cylinder set $[\pmb\omega]$, $\tilde\nu_n(\sigma^{-k}[\pmb\omega])=\tilde\nu_n([\pmb{p}(\omega_1)\pmb\omega])=\tilde\nu_n([\pmb{p}(\omega_1)\pmb\omega \pmb{s}(\omega_{n-k})])=\mu([\pmb\omega])$, hence for $i=k,\ldots, n-k+1$ we have
\[
\left|\int w d(\mu\circ\sigma^{-i}) - \int w d(\tilde\nu_n\circ\sigma^{-i})\right| \leq \var_{i-k}(w).
\]
The other summands can be estimated from above by $\var_0(w)$. Summarizing,
\[
\left|\int w d\mu - \int w d\hat\nu_n\right| \leq \frac {2k} n \var_0(w) + \frac {n-2k}n \frac 1 {n-2k} \sum_{i=1}^{n-2k} \var_i(w) \to 0.
\]
\end{proof}
The combination of Claims \ref{claim:1} and \ref{claim:2} proves the proposition.
\end{proof}

\begin{proof}[Proof of Theorem~\ref{thm:markovhaus}] The strategy of the proof is the following:
\begin{enumerate}
  \item Find a $\sigma$-invariant ergodic probability measure $\mu$ on $\Sigma_A$ which natural projection is a candidate for achieving the Hausdorff dimension;
  \item find a approximating sequence of $n$-step Bernoulli measures $\hat{\nu}_n$ such that $\hat{\nu}_n\to\mu$ in weak-* and entropy topology;
  \item show that $\dim_H\Pi_*\hat{\nu}_n\to s$ as $n\to\infty$.
\end{enumerate}

First, we find the measure $\mu$. Let $s$ be such that $\rho(A^{(s)})=1$. Since there exists a $k\geq1$ such that $(A^{(s)})^k$ has strictly positive elements. Then by Perron-Frobenius Theorem, there exists a vector $p=(p_1,\ldots,p_m)^T$ with strictly positive elements such that $A^{(s)}p=p$. Let $P_{i,j}=A^{(s)}_{i,j}\frac{p_j}{p_i}$. Then the matrix $P=(P_{i,j})_{i,j=1}^m$ is a probability matrix, which is aperiodic and recurrent. Thus, there exists a unique probability vector $q=(q_1,\ldots,q_m)$ with positive elements such that $qP=q$. Then for a cylinder set $[i_1,\ldots,i_n]$ let
\begin{equation}\label{eq:markovmeas}
\mu([i_1,\ldots,i_n])=q_{i_1}P_{i_1,i_2}\cdots P_{i_{n-1},i_n}.
\end{equation}
It is easy to see by the definition of Lyapunov exponents in formula
\cite[(8.1)]{BaRaSiKochin1}
that
$$
h(\mu)=-\sum_{i,j}q_iP_{i,j}\log P_{i,j}=-\sum_{i=1}^mq_i\log|\alpha_i|\gamma_i^{-(s-1)}=\chi_2(\mu)+(s-1)\chi_1(\mu).
$$
Moreover, since $\frac{h(\mu)}{\chi_1(\mu)}\leq1< s$ we have $\chi_2(\mu)<\chi_1(\mu)$, and thus, $D(\mu)=s$ by
\cite[Definition~8.2]{BaRaSiKochin1}.

By Proposition~\ref{prop:nbern}, for every $\varepsilon>0$ there exists a sequence of $n$-step Bernoulli measures $\hat{\nu}_n$ and a $N\geq1$ such that for every $n\geq N$
$$
|h(\mu)-h(\hat{\nu}_n)|,|\chi_2(\mu)-\chi_2(\hat{\nu}_n)|,|\chi_1(\mu)-\chi_1(\hat{\nu}_n)|<\varepsilon.
$$
One can choose $\varepsilon<(\chi_1(\mu)-\chi_2(\mu))/100$, so $\chi_2(\hat{\nu}_n)<\chi_1(\hat{\nu}_n)$. Now, we approximate $\hat{\nu}_n$ with a $nm$-step Bernoulli measure $\overline{\nu}_{n,m}$, which is supported on words $\pmb\omega\in(\Sigma_{A,n})^m$ for which $\gamma_{\pmb\omega}^{-1}<\alpha_{\pmb\omega}$. More precisely, let
$$
Y_{m,n}=\{\pmb\omega\in\Sigma_{A,nm}:\hat{\nu}_n(C[\omega])>0\text{ and }\gamma_{\pmb\omega}^{-1}<\alpha_{\pmb\omega}\},
$$
and let $\hat{\nu}_{n,m}$ be the Bernoulli measure on $(Y_{m,n})^\N$ defined with the probabilities $\left(\hat{\nu}_n(C[\pmb\omega])/\hat{\nu}_n(Y_{m,n})\right)_{\pmb\omega\in Y_{m,n}}$, and let $\overline{\nu}_{n,m}$ be the corresponding $nm$-step Bernoulli measure.

By the strong law of large numbers and Egorov's Theorem, for every $\varepsilon>0$ there exists $M=M(n)>0$ such that for every $m\geq M$
$$
|h(\overline{\nu}_{n,m})-h(\hat{\nu}_n)|,|\chi_2(\overline{\nu}_{n,m})-\chi_2(\hat{\nu}_n)|,|\chi_1(\overline{\nu}_{n,m})-\chi_1(\hat{\nu}_n)|<\varepsilon.
$$
Thus, $|s-D(\overline{\nu}_{n,m})|<C\varepsilon$ with some constant $C>0$ independent of $n,m$.

By definition, $\mathrm{supp}(\Pi_*\overline{\nu}_{n,m})\subseteq\mathrm{graph}(G)$. Thus, in order to apply 
\cite[Theorem~6.3]{BaRaSiKochin1}, it is enough to show that there exists $\pmb\omega\neq\pmb\tau\in Y_{m,n}$ such that $D\tilde{F}_{\pmb\omega}$ and $D\tilde{F}_{\pmb\tau}$ are not simultaneously diagonalisable. Let $\ell\geq1$ and $\pmb\omega_1,\pmb\tau_1\in\Sigma_{A,\ell}$ as in \eqref{eq:assum}. Without loss of generality we may assume that $n-2k\gg\ell$. Since the first and last symbols of $\pmb\omega_1,\pmb\tau_1$ are the same, one can choose $\pmb\upsilon_1,\pmb\upsilon_2$ such that
$\hat{\nu}_n(C[\pmb\upsilon_1\pmb\omega_1\pmb\upsilon_2]),\hat{\nu}_n(C[\pmb\upsilon_1\pmb\tau_1\pmb\upsilon_2])>0$. By the strong law of large numbers, for every sufficiently large $m\geq1$ one can find $\pmb\kappa\in\Sigma_{A,n(m-1)}$ such that
$\pmb\upsilon_1\pmb\omega_1\pmb\upsilon_2\pmb\kappa,\pmb\upsilon_1\pmb\tau_1\pmb\upsilon_2\pmb\kappa\in Y_{m,n}$. By definition, $\alpha_{\pmb\upsilon_1\pmb\tau_1\pmb\upsilon_2\pmb\kappa}=\alpha_{\pmb\upsilon_1\pmb\omega_1\pmb\upsilon_2\pmb\kappa}$ and $\gamma_{\pmb\upsilon_1\pmb\tau_1\pmb\upsilon_2\pmb\kappa}=\gamma_{\pmb\upsilon_1\pmb\omega_1\pmb\upsilon_2\pmb\kappa}$. Thus, $D\tilde{F}_{\pmb\upsilon_1\pmb\tau_1\pmb\upsilon_2\pmb\kappa}$ and $D\tilde{F}_{\pmb\upsilon_1\pmb\omega_1\pmb\upsilon_2\pmb\kappa}$ are not simultaneously diagonalisable if and only if $D\tilde{F}_{\pmb\upsilon_1\pmb\tau_1\pmb\upsilon_2\pmb\kappa}\neq D\tilde{F}_{\pmb\upsilon_1\pmb\omega_1\pmb\upsilon_2\pmb\kappa}$. But this is true since $D\tilde{F}_{\pmb\omega_1}\neq D\tilde{F}_{\pmb\tau_1}$. Hence, by 
\cite[Theorem~6.3]{BaRaSiKochin1}
$$
\dim_H\mathrm{graph}(G)\geq\dim_H\Pi_*\overline{\nu}_{n,m}=D(\overline{\nu}_{n,m})\geq s-C\varepsilon.
$$
The statement follows by taking $\varepsilon\to0$.
\end{proof}

\begin{proof}[Proof of Theorem~\ref{thm:markovbox}]
	Since the lower box-counting dimension is always an upper bound for the Hausdorff dimension and the upper box counting dimension is always at most $s$, in point of view of Theorem~\ref{thm:markovhaus}, it is enough to show for diagonal systems. That is, by applying an affine transformation on the dataset $\Delta$, we may assume that $a_i=0$ for every $i=1,\ldots,m$. Since $\Delta$ is not collinear, $G([0,1])$ is an interval $D$ with $|D|>0$. Let $\Sigma_A^{(r)}=\left\{\pmb\omega\in\bigcup_{\ell=1}^\infty\Sigma_{A,\ell}:\gamma_{\pmb\omega}^{-1}\leq r<\gamma_{\pmb\omega|_{|\pmb\omega|-1}}^{-1}\right\}$. There needed at least $\sum_{\pmb\omega\in\Sigma_A^{(r)}}\left\lceil\frac{|D|\cdot\alpha_{\pmb\omega}}{\gamma_{\pmb\omega}^{-1}}\right\rceil$-many squares of side length $r$ to cover $\mathrm{graph}(G)$. By using the measure $\mu$ defined in \eqref{eq:markovmeas},
$$
\sum_{\pmb\omega\in\Sigma_A^{(r)}}\left\lceil\frac{|D|\cdot\alpha_{\pmb\omega}}{\gamma_{\pmb\omega}^{-1}}\right\rceil\geq r^{-s}\sum_{\pmb\omega\in\Sigma_A^{(r)}}\frac{|D|\cdot\alpha_{\pmb\omega}}{\gamma_{\pmb\omega}^{-1}}\gamma_{\pmb\omega}^{-s}\geq r^{-s}C\sum_{\pmb\omega\in\Sigma_A^{(r)}}\mu([\pmb\omega])=r^{-s}C,
$$
where $C=|D|\min_{i,j}p_i/p_j$.
\end{proof}

\section{Continuous Generalized Takagi Functions}

In the previous examples, the base dynamics $f\colon[0,1]\mapsto[0,1]$, was a Markovian expanding, piecewise linear map with Markov partition formed by intervals. For general systems of the form \eqref{eq:interpoldyn}, the base dynamics is not Markovian. However, it is hard to get a graph of a continuous function as a repeller of such systems. Finally, we present here a special case, for which the repeller is a continuous function graph but the base dynamics is non-markovian. This example can be considered as generalised Takagi functions.

Let us recall that the $\alpha$-Takagi function $T_\alpha$ was defined as
$T_{\alpha}(x):=\sum\limits_{n=1}^{\infty }\alpha^n\psi\left(2^n \cdot x\right)$, where we defined $\psi(z)=\mathrm{dist}\left(z,\mathbb{Z}\right)$.

To define a continuous generalization of this family first we fix the two parameters $\alpha\in\left(0,1\right)$ and $\beta\in\left(1,2\right)$ such that
$\alpha \cdot \beta>1$. Then we introduce (see Figure \ref{d87}) the function $B_\beta:[0,1]\to[0,1]$
\begin{equation}\label{d98}
B_\beta(x):=\left\{
\begin{array}{ll}
\beta x, & \hbox{if $x\in\left[0,\frac{1}{2}\right]$;} \\
1-\beta(1-x), & \hbox{if $x\in\left(\frac{1}{2},1\right]$.}
\end{array}
\right.
\end{equation}
This map will be our base dynamics.

\begin{figure}
	\centering
	\includegraphics[width=11cm]{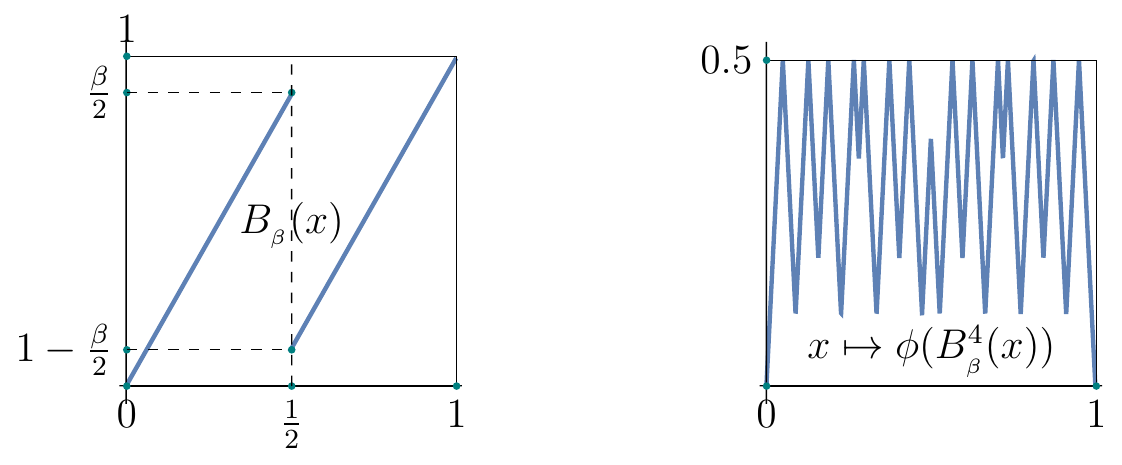}
	\caption{The functions $B_\beta(x)$ and $\psi\left(B_{\beta}^{4}(x)\right)$}\label{d87}
\end{figure}

Now we define the continuous generalized $(\alpha,\beta)$-Takagi function
$ T_{\alpha,\beta}\colon[0,1]\to\mathbb{R}^+$
as
\begin{equation}\label{d86}
T_{\alpha,\beta}(x):= \sum\limits_{k=0}^{\infty }
\alpha^k \cdot \psi\left(B_{\beta}^{n}(x)\right).
\end{equation}

\begin{figure}
	\centering
	\includegraphics[width=12cm]{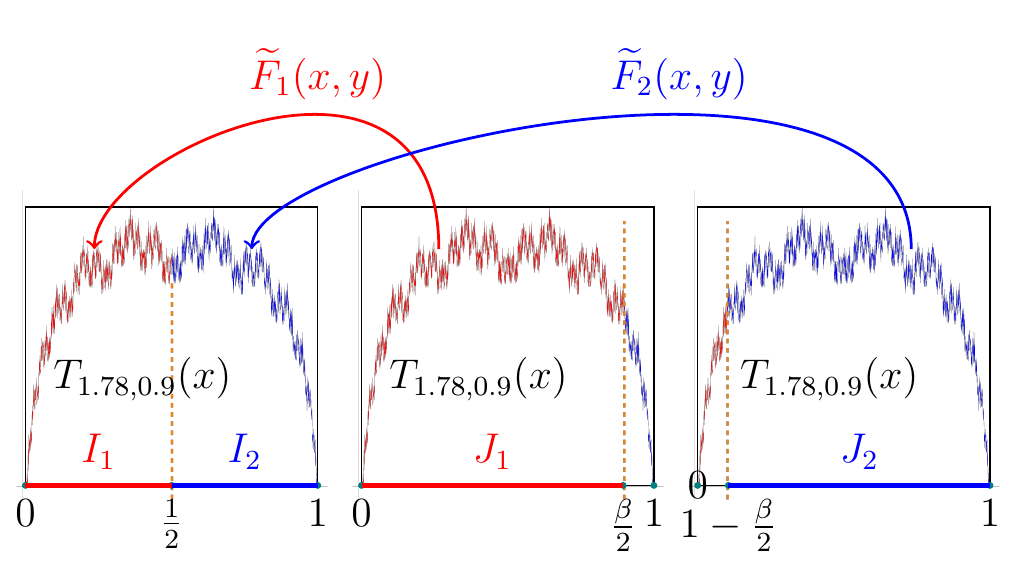}
	\caption{$\mathrm{graph}(T_{1.78,0.9})$ is the union of affine  images of \emph{parts} of  $\mathrm{graph}(T_{1.78,0.9})$}\label{d80}
\end{figure}

The fact that the function $ T_{\alpha,\beta}(x)$ is continuous follows from the
fact that for all $n$ the function $x\mapsto \psi\left(B_{\beta}^{n}(x)\right)$
is continuous (see the right-hand side of Figure \ref{d87}). Indeed, it is easy to see by the symmetry $B_\beta(x)=1-B_\beta(1-x)$ that for a continuous function $g\colon[0,1]\mapsto\R$, which is symmetric to the line $x=1/2$, the map $g\circ B_\beta$ is continuous and symmetric to $x=1/2$.

The graphs of the functions $T_{\alpha,\beta}(x)$ are not self-affine but the graphs of these functions have a less restrictive weakened form of self-affinity. Namely,
we write
\begin{equation}\label{d85}
I_1:=\left[0,\frac{1}{2}\right],\
I_2:=\left[\frac{1}{2},1\right] \mbox{ and }
J_1:=\left[0,\frac{\beta}{2}\right], \
J_2:=\left[1-\frac{\beta}{2}.1\right]
\end{equation}
and
\begin{equation}\label{d83}
\widetilde{I}_\ell :=I_\ell \times[0,M_{\alpha,\beta}] \mbox{ and }
\widetilde{J}_\ell :=J_\ell \times[0,M_{\alpha,\beta}],\quad \ell =1,2,
\end{equation}
where $M_{\alpha,\beta}:=\max\limits_{x\in[0,1]}T_{\alpha,\beta}(x)$. Then
\begin{equation}\label{d84}
\mathrm{Graph}(T_{\alpha,\beta})=
\widetilde{F}_1\left(\mathrm{Graph}(T_{\alpha,\beta})\cap\widetilde{J}_1\right)\bigcup
\widetilde{F}_2\left(\mathrm{Graph}(T_{\alpha,\beta})\cap\widetilde{J}_2\right),
\end{equation}
where
\begin{equation}\label{d82}
\begin{split}
\widetilde{F}_0(x,y)&:=\left(\frac{1}{\beta} \cdot x,
\frac{1}{\beta} \cdot x+\alpha \cdot y
\right)
\mbox{ and }\\
\widetilde{F}_1(x,y)&:=\left(1-\frac{1}{\beta} \cdot (1- x),
\frac{1}{\beta} \cdot (1- x)+\alpha \cdot y
\right).
\end{split}
\end{equation}
The union in \eqref{d84} is almost disjoint, the intersection is the only point of
$\mathrm{graph}(T_{\alpha,\beta})$ which lies on the vertical line $x=\frac{1}{2}$. This follows from the fact that
\begin{equation}\label{d81}
\mathrm{graph}(T_{\alpha,\beta})\cap\widetilde{I}_\ell =
\widetilde{F}_\ell \left(\mathrm{Graph}(T_{\alpha,\beta})\cap\widetilde{J}_\ell \right), \quad \ell =1,2.
\end{equation}
(See Figure \ref{d80}.)
If we compare this function graph with the graph of the self affine Takagi map $T_{3/2}$ (see on the right-hand side of the Figure
\ref{fig:takwei}) then we can see the difference. Namely, in the case of
$T_{3/2}$, both the left- and right-hand sides of $\mathrm{graph}(T_\alpha)$
are affine images of the whole graph  $\mathrm{graph}(T_\alpha)$.
As opposed to that
in the case of $T_{\alpha,\beta}$ the left and the right hand sides:
$\mathrm{graph}(T_{\alpha,\beta})\cap \widetilde{I}_1$ and
$\mathrm{graph}(T_{\alpha,\beta})\cap \widetilde{I}_2$ are affine images of certain \emph{parts} of $\mathrm{graph}(T_{\alpha,\beta})$ and not the whole one. That is why the family of $T_{\alpha,\beta}$ is much more general.

\begin{theorem}
	For every value of $\alpha$ and $1<\beta\leq 2$ such that $\alpha\cdot\beta>1$
	$$
	\dim_H\mathrm{graph}(T_{\alpha,\beta})=\dim_B\mathrm{graph}(T_{\alpha,\beta})=2+\frac{\log\alpha}{\log\beta}.
	$$
\end{theorem}

In order to calculate $\dim_H\mathrm{graph}(T_{\alpha,\beta})$, we give the upper bound by using natural covers and for the lower bound we find "large enough" Markovian subsystems of $B_\beta$. The set of admissible sequences is $$\Sigma_\beta=\{(i_1,i_2,\ldots):\exists x\in[0,1]\text{ such that }B_\beta^n(x)\in I_{i_n}\text{ for every }n\geq1\}.$$ Since the base system $B_\beta$ is not Markovian for a general value of $\beta$, the set of admissible sequences cannot be generated by an adjacency matrix. By Rokhlin's formula, see \cite{Rohlinent, Ledent}, $\lim_{n\to\infty}\frac{1}{n}\log\sharp\Sigma_\beta^{(n)}=\log\beta$, where $\Sigma_\beta^{(n)}=\{(i_1,\ldots,i_n):\exists\jj\in\Sigma_\beta\text{ such that }j_k=i_k\text{ for }k\geq1\}$.

For each ${\pmb\omega}\in\Sigma_\beta^{(n)}$, let us define the cylinder sets by induction. Namely, for $n=1$ let $\mathcal{C}_\omega=\tilde{F}_\omega(\widetilde{J}_\omega)$ the cylinder set corresponding to $\omega\in\Sigma_\beta^{(1)}$. For $n>1$ and $\pmb\omega\in\Sigma_\beta^{(n)}$, let  $\mathcal{C}_{\pmb\omega}=\tilde{F}_{\omega_1}(\mathcal{C}_{\sigma\pmb\omega}\cap\widetilde{J}_{\omega_1})$, where $\sigma{\pmb\omega}$ is the word of lenght $n-1$ by deleting the first symbol of $\pmb\omega$. For each $\pmb\omega\in\Sigma_\beta^{(n)}$, the set $\mathcal{C}_{\omega}$ is a parallelogram with height parallel to the $x$-axis is at most $\beta^{-n}$ and side length parallel to the $y$-axis is $\alpha^{n}M_{\alpha,\beta}$. Since $\alpha\beta>1$ we get that the tangent of the angle between the sides is uniformly bounded, denote the bound by $C$. Thus,  $\mathrm{graph}(T_{\alpha,\beta})$ can be covered by at most $\sharp\Sigma_\beta^{(n)}\cdot (M_{\alpha,\beta}(\alpha\beta)^{n}+C)$-many squares of sidelength $\beta^{-n}$. This shows that
$$
\overline{\dim}_B\mathrm{graph}(T_{\alpha,\beta})\leq2+\frac{\log\alpha}{\log\beta}.
$$


Now, we introduce the Markovian subsystems of $B_\beta$. A compact $B_\beta$-invariant set $\mathcal{B}$ is called \emph{Markov subset} if there exists a finite collection $\DD$ of closed intervals such that for every $\mathfrak{I}_1,\mathfrak{I}_2\in\DD$.
\begin{enumerate}
	\item $\mathfrak{I}_1\subseteq I_1$ or $\mathfrak{I}_1\subseteq I_2$,
	\item $\mathfrak{I}_1^o\cap\mathfrak{I}_2^o=\emptyset$ if $\mathfrak{I}_1\neq\mathfrak{I}_2$,
	\item $\bigcup_{\mathfrak{I}\in\DD}\cap\mathcal{B}=\mathcal{B}$,
	\item either $B_\beta(\mathfrak{I}_1\cap\mathcal{B})\cap\mathfrak{I}_2\cap\mathcal{B}=\emptyset$ or $\mathfrak{I}_2\cap\mathcal{B}\subseteq\mathcal{B}_\beta(\mathfrak{I}_1\cap\mathcal{B})$
\end{enumerate}
We call $\DD$ the Markov partition of $\mathcal{B}$. Now we show that there exist a sequence of Markov subsystems, which topologycal entropy approximates $\log\beta$ arbitrarily.

\begin{lemma}
	For every $\varepsilon>0$ there exists $m\geq1$, a Markov subset $\mathcal{B}_m\subset[0,1]$ and $\DD_m$ Markov partition such that
		$$h_{top}(B_\beta|_{\mathcal{B}_m})>h_{top}(B_\beta)-\varepsilon.$$
		Moreover, we can assume that there exist intervals in $\DD_m$ which contain $0$ and $1$.
\end{lemma}

The claim follows from Hofbauer, Raith and Simon \cite[Proposition~1(a),(b),(c) and Lemma~2]{hofraithsim}.

Similarly to \eqref{eq:smatrix}, we define a matrix $A^{(s),m}$ for every $m\geq1$, which gives the dimension of $\mathrm{graph}(T_{\alpha,\beta}|_{\mathcal{B}_m})$. Namely, let $A^{(s),m}$ be a $\#\DD_m\times\#\DD_m$ matrix such that
$$
A^{(s),m}_{\mathfrak{I},\mathfrak{J}}=\begin{cases}\alpha\beta^{-(s-1)} & \text{if }\mathfrak{J}\cap\mathcal{B}_m\subseteq B_{\beta}(\mathfrak{I}\cap\mathcal{B}_m)\text{ for }\mathfrak{I},\mathfrak{J}\in\DD_m \\
0 & \text{otherwise}.\end{cases}
$$

Let $s_m$ be such that $\rho(A^{(s_m),m})=1$. For, $\mathfrak{I},\mathfrak{J}\in\DD_m$, let
\begin{multline*}
\mathfrak{I}\stackrel{n}{\rightarrow}\mathfrak{J}=\left\{(\mathfrak{I}_1,\ldots,\mathfrak{I}_{n}):\right.\\
\left.\mathfrak{I}_j\in\DD_m,\ \mathfrak{I}_1=\mathfrak{I}, \mathfrak{I}_n=\mathfrak{J},\ B_\beta(\mathfrak{I}_j\cap\mathcal{B}_m)\supseteq\mathfrak{I}_{j+1}\cap\mathcal{B}_m\text{ for }1\leq j\leq n-1\right\}.
\end{multline*}
By definition,
$$
h_{top}(B_\beta|_{\mathcal{B}_m})=\lim_{n\to\infty}\frac{\log\#\bigcup_{\mathfrak{I},\mathfrak{J}\in\DD_m}\mathfrak{I}\stackrel{n}{\rightarrow}\mathfrak{J}}{n}.
$$
But for every $k\geq 1$, and $\mathfrak{I},\mathfrak{J}\in\DD_m$,
$$
\left(\left(A^{(s_m)}\right)^k\right)_{\mathfrak{I},\mathfrak{J}}=\left(\alpha\beta^{-(s_m-1)}\right)^k\cdot\#(\mathfrak{I}\stackrel{n}{\rightarrow}\mathfrak{J}).
$$
Hence,
$$
\log\beta-\varepsilon<h_{top}(B_\beta|_{\mathcal{B}_m})=\lim_{k\to\infty}\frac{\log\frac{\|\left(A^{(s_m)}\right)^k\|_1}{\left(\alpha\beta^{-(s_m-1)}\right)^k}}{k}=-\log\left(\alpha\beta^{-(s_m-1)}\right),
$$
which implies that $s_m>2+\frac{\log\alpha}{\log\beta}-\varepsilon.$ One can decompose $\DD_m$ into recurrent and transient classes. It is easy to see that there exists a recurrent class $R$ such that restricting $A^{(s_m),m}$ for $R$, $\rho(A^{(s_m),m}|_R)=1$. Denote the Markov subset of $\mathcal{B}_m$ restricted to $R$ by $\mathcal{R}_m$. Similarly to \eqref{eq:markovmeas}, there exists a Markov measure $\mu_m$ such that $D(\mu_m)=s_m$. By Proposition~\ref{prop:nbern}, for every $\varepsilon>0$ there exists an $n$-step Bernoulli measure $\nu_{n,m}$ such that $D(\nu_{n,m})>s_m-\varepsilon$. By 
\cite[Theorem~7.6]{BaRaSiKochin1}
, $\dim_H\Pi_*\nu_{n,m}=D(\nu_{n,m})$, which gives the lower bound.

\bibliographystyle{plain}
\bibliography{barnsley_08_08_bb}

\begin{thebibliography}{10}

\bibitem{AllKawa}
Pieter~C. Allaart and Kiko Kawamura.
\newblock The {T}akagi function: a survey.
\newblock {\em Real Anal. Exchange}, 37(1):1--54, 2011/12.

\bibitem{baranski}
Krzysztof Bara\'{n}ski.
\newblock Dimension of the graphs of the {W}eierstrass-type functions.
\newblock In {\em Fractal geometry and stochastics {V}}, volume~70 of {\em
  Progr. Probab.}, pages 77--91. Birkh\"{a}user/Springer, Cham, 2015.

\bibitem{BBR}
Krzysztof Bara\'{n}ski, Bal\'{a}zs B\'{a}r\'{a}ny, and Julia Romanowska.
\newblock On the dimension of the graph of the classical {W}eierstrass
  function.
\newblock {\em Adv. Math.}, 265:32--59, 2014.

\bibitem{BHR}
Bal\'azs B{\'a}r{\'a}ny, Michael {Hochman}, and Ariel Rapaport.
\newblock {Hausdorff dimension of planar self-affine sets and measures}.
\newblock {\em ArXiv preprint arXiv:1712.07353}, December 2017.

\bibitem{BaRaSiKochin1}
Bal\'{a}zs B{\'a}r{\'a}ny, Michal Rams, and K\'{a}roly. Simon.
\newblock Dimension theory of some non-{M}arkovian repellers {P}art {I}: {A}
  gentle introduction.

\bibitem{Barnsley}
Michael~F. Barnsley.
\newblock Fractal functions and interpolation.
\newblock {\em Constr. Approx.}, 2(4):303--329, 1986.

\bibitem{BarnsleyHarrington}
Michael~F. Barnsley and Andrew~N. Harrington.
\newblock The calculus of fractal interpolation functions.
\newblock {\em J. Approx. Theory}, 57(1):14--34, 1989.

\bibitem{Bedford}
Tim Bedford.
\newblock H\"{o}lder exponents and box dimension for self-affine fractal
  functions.
\newblock {\em Constr. Approx.}, 5(1):33--48, 1989.
\newblock Fractal approximation.

\bibitem{BesUr}
A.~S. Besicovitch and H.~D. Ursell.
\newblock Sets of fractional dimensions (v): on dimensional numbers of some
  continuous curves.
\newblock {\em Journal of the London Mathematical Society}, s1-12(1):18--25.

\bibitem{Billi}
Patrick Billingsley.
\newblock Notes: {V}an {D}er {W}aerden's {C}ontinuous {N}owhere
  {D}ifferentiable {F}unction.
\newblock {\em Amer. Math. Monthly}, 89(9):691, 1982.

\bibitem{Erdos1}
Paul Erd\"{o}s.
\newblock On a family of symmetric {B}ernoulli convolutions.
\newblock {\em Amer. J. Math.}, 61:974--976, 1939.

\bibitem{Erdos2}
Paul Erd\"{o}s.
\newblock On the smoothness properties of a family of {B}ernoulli convolutions.
\newblock {\em Amer. J. Math.}, 62:180--186, 1940.

\bibitem{FalconerMiao}
Kenneth Falconer and Jun Miao.
\newblock Dimensions of self-affine fractals and multifractals generated by
  upper-triangular matrices.
\newblock {\em Fractals}, 15(3):289--299, 2007.

\bibitem{Hardy}
G.~H. Hardy.
\newblock Weierstrass's non-differentiable function.
\newblock {\em Trans. Amer. Math. Soc.}, 17(3):301--325, 1916.

\bibitem{hochman2012self}
Michael Hochman.
\newblock On self-similar sets with overlaps and inverse theorems for entropy.
\newblock {\em Ann. of Math. (2)}, 180(2):773--822, 2014.

\bibitem{hofraithsim}
Franz Hofbauer, Peter Raith, and K\'aroly Simon.
\newblock {H}ausdorff dimension for some hyperbolic attractors with overlaps
  and without finite {M}arkov partition.
\newblock {\em Ergodic Theory Dynam. Systems}, 27(4):1143--1165, 2007.

\bibitem{KMY}
James~L. Kaplan, John Mallet-Paret, and James~A. Yorke.
\newblock The {L}yapunov dimension of a nowhere differentiable attracting
  torus.
\newblock {\em Ergodic Theory Dynam. Systems}, 4(2):261--281, 1984.

\bibitem{Ledent}
Fran\c{c}ois Ledrappier.
\newblock Some properties of absolutely continuous invariant measures on an
  interval.
\newblock {\em Ergodic Theory Dynamical Systems}, 1(1):77--93, 1981.

\bibitem{Ledrapp}
Fran\c{c}ois Ledrappier.
\newblock On the dimension of some graphs.
\newblock In {\em Symbolic dynamics and its applications ({N}ew {H}aven, {CT},
  1991)}, volume 135 of {\em Contemp. Math.}, pages 285--293. Amer. Math. Soc.,
  Providence, RI, 1992.

\bibitem{Mandel}
Benoit~B. Mandelbrot.
\newblock {\em Fractals: form, chance, and dimension}.
\newblock W. H. Freeman and Co., San Francisco, Calif., revised edition, 1977.
\newblock Translated from the French.

\bibitem{Rohlinent}
V.~A. Rohlin.
\newblock Exact endomorphisms of a {L}ebesgue space.
\newblock {\em Izv. Akad. Nauk SSSR Ser. Mat.}, 25:499--530, 1961.

\bibitem{RuanSuYao}
Huo-Jun Ruan, Wei-Yi Su, and Kui Yao.
\newblock Box dimension and fractional integral of linear fractal interpolation
  functions.
\newblock {\em J. Approx. Theory}, 161(1):187--197, 2009.

\bibitem{Shen}
Weixiao Shen.
\newblock Hausdorff dimension of the graphs of the classical {W}eierstrass
  functions.
\newblock {\em Math. Z.}, 289(1-2):223--266, 2018.

\bibitem{Solomyak98}
Boris Solomyak.
\newblock Measure and dimension for some fractal families.
\newblock {\em Math. Proc. Cambridge Philos. Soc.}, 124(3):531--546, 1998.

\bibitem{takagi}
T.~Takagi.
\newblock A simple example of the continuous function without derivative.
\newblock {\em Tokyo Sugaku-Butsurigakkwai Hokoku}, 1:F176--F177, 1901.

\bibitem{varju2018dimension}
P{\'e}ter~P Varj{\'u}.
\newblock On the dimension of bernoulli convolutions for all transcendental
  parameters.
\newblock {\em arXiv preprint arXiv:1810.08905}, 2018.

\end{thebibliography}

\end{document}